\documentclass[a4paper, 12pt]{amsart}
\usepackage[english]{babel}
\usepackage{amsmath,amssymb,amsthm}

\theoremstyle{plain}
\newtheorem{theorem}{Theorem}[section]
\newtheorem{corollary}[theorem]{Corollary}
\newtheorem{prop}[theorem]{Proposition}
\newtheorem{lemma}[theorem]{Lemma}

\theoremstyle{definition}
\newtheorem{remarks}[theorem]{Remarks}
\newtheorem{remark}[theorem]{Remark}
\newtheorem{fact}[theorem]{Fact}

\newtheorem{definition}[theorem]{Definition}

 \DeclareMathOperator{\e}{\mathrm{e}}
 \DeclareMathOperator{\Id}{\mathrm{Id}}
 \DeclareMathOperator{\Clop}{\mathrm{Clop}}

\newcommand{\tri} {\vert \hspace{-1.5pt} \vert  \hspace{-1.5pt} \vert }

\newcommand{\C}{\mathbb{C}}
\newcommand{\R}{\mathbb{R}}
\newcommand{\N}{\mathbb{N}}

\newcommand{\eps}{\varepsilon}

\renewcommand{\leq}{\leqslant}
\renewcommand{\geq}{\geqslant}

\newcounter{equi1}

\begin{document}
\begin{center}
{\tiny  J.~Math.~Anal.~Appl.\ (2008),
\texttt{doi:10.1016/j.jmaa.2008.04.021} }
\end{center}
\title[Extremely non-complex $C(K)$ spaces]{Extremely non-complex $\boldsymbol{C(K)}$ spaces}

\author{Piotr Koszmider}

\address{Instytut Matematyki Politechniki \L\'odzkiej,
ul. W\'olcza\'nska 215, 90-924 \L\'od\'z, Poland}
\email{\texttt{pkoszmider@im0.p.lodz.pl}}

\author{Miguel Mart\'{\i}n}
\author{Javier Mer\'{\i}}
 \address{Departamento de An\'{a}lisis Matem\'{a}tico \\ Facultad de
 Ciencias \\ Universidad de Granada \\ 18071 Granada, Spain}
 \email{\texttt{mmartins@ugr.es, jmeri@ugr.es}}
\thanks{Research partially supported by Spanish MEC project no.\
 MTM2006-04837 and Junta de Andaluc\'{\i}a grants FQM-185 and P06-FQM-01438}

\date{December 19th, 2007}

\dedicatory{To Isaac Namioka on his 80th birthday}

  \thispagestyle{empty}

\begin{abstract}
We show that there exist infinite-dimensional extremely
non-complex Banach spaces, i.e.\ spaces $X$ such that the norm
equality $\|\Id + T^2\|=1 + \|T^2\|$ holds for every bounded
linear operator $T:X\longrightarrow X$. This answers in the
positive Question~4.11 of [Kadets, Mart\'{\i}n, Mer\'{\i}, Norm
equalities for operators, \emph{Indiana U.\ Math.\
J.}~\textbf{56} (2007), 2385--2411]. More concretely, we show
that this is the case of some $C(K)$ spaces with few operators
constructed in [Koszmider, Banach spaces of continuous
functions with few operators, \emph{Math.\ Ann.}~\textbf{330}
(2004), 151--183] and [Plebanek, A construction of a Banach
space $C(K)$ with few operators, \emph{Topology
Appl.}~\textbf{143} (2004), 217--239]. We also construct
compact spaces $K_1$ and $K_2$ such that $C(K_1)$ and $C(K_2)$
are extremely non-complex, $C(K_1)$ contains a complemented
copy of $C(2^\omega)$ and $C(K_2)$ contains a (1-complemented)
isometric copy of $\ell_\infty$.
\end{abstract}

\maketitle

\section{Introduction}

Let $X$ be a real Banach space. We write $L(X)$ for the space
of all bounded linear operators on $X$ endowed with the
supremum norm and $W(X)$ for its subspace of all weakly compact
operators on $X$.

The aim of this paper is to show that there are
infinite-dimensional Banach spaces $X$ for which the norm
equality
\begin{equation}\label{eq:square}\tag{\textrm{sDE}}
\|\Id + T^2\|=1 + \|T^2\|
\end{equation}
holds for every $T\in L(X)$. Actually, we will show that there
are several different compact (Hausdorff) spaces $K$ such that
the corresponding real spaces $C(K)$ satisfy this property.
This answers positively Question~4.11 of the very recent paper
\cite{K-M-M} by V.~Kadets and the second and third authors.

Let us motivate the interest of the question.

Firstly, a good interpretation of the property we are dealing
with is given by the so-called complex structures on real
Banach spaces. We recall that a (real) Banach space $X$ is said
to have a \emph{complex structure} if there exists $T\in L(X)$
such that $T^2=-\Id$. This allows us to define on $X$ a
structure of vector space over $\C$, by setting
$$
(\alpha + i \beta)x = \alpha x + \beta T(x) \qquad \bigl(\alpha+i\beta \in \C,
\ x\in X\bigr).
$$
Moreover, by just defining
$$
\tri x\tri =\max\bigl\{\|\e^{i\theta}x\|\ : \ \theta \in[0,2\pi]\bigr\} \qquad (x\in X)
$$
one gets a complex norm on $X$ which is equivalent to the original
one. Conversely, if $X$ is the real space underlying a complex
Banach space, then the bounded linear operator defined by
$T(x)=i\,x$ for every $x\in X$, satisfies that $T^2=-\Id$. In the
finite-dimensional setting, complex structures appear if and only if
the dimension of the space is even. In the infinite-dimensional
setting, there are real Banach spaces admitting no complex
structure. This is the case of the James' space $\mathcal{J}$ (see
\cite[\S 3.4]{Albiac-Kalton} for the definition), as it was shown by
J.~Dieudonn\'{e} in 1952 \cite{Dieu1952}. More examples of this kind
have been constructed over the years, including uniformly convex
examples (S.~Szarek 1986 \cite{Szarek-PAMS86}), the hereditary
indecomposable space of T.~Gowers and B.~Maurey \cite{GowMau1993}
or, more generally, any space such that every operator on it is a
strictly singular perturbation of a multiple of the identity. Gowers
also constructed a space of this kind with an unconditional basis
\cite{Gower1994,GowMau1997}. We refer the reader to the very recent
papers by V.~Ferenczi and E.~Medina~Galego
\cite{Ferenczi-Adv,FerenMedina} and references therein for a
discussion about complex structures on spaces and on their
hyperplanes.

Let us comment that if equation \eqref{eq:square} holds for all
operators on a Banach space $X$, then $X$ does not have complex
structure in the strongest possible way, showing that, for every
$T\in L(X)$, the distance from $T^2$ to $-\Id$ is the biggest
possible, namely $1 + \|T^2\|$. This observation justifies the
following definition which we will use along the paper.

\begin{definition}
We say that $X$ is \emph{extremely non-complex} if the norm
equality
\begin{equation*}\tag{\textrm{\ref{eq:square}}}
\|\Id + T^2\|=1 + \|T^2\|
\end{equation*}
holds for every $T\in L(X)$.
\end{definition}

Secondly, let us explain shortly the history leading to the
appearance of \eqref{eq:square} and the question of the
existence of infinite-dimensional extremely non-complex spaces
in the already cited paper \cite{K-M-M}. The interest in norm
equalities for operators goes back to the 1960's, when
I.~Daugavet \cite{Dau} showed that each compact operator $T$ on
$C[0,1]$ satisfies the norm equality
\begin{equation*}\label{DE}\tag{DE}
\|\Id + T\|= 1 + \|T\|.
\end{equation*}
The above equation is nowadays referred to as \emph{Daugavet
equation}. This result has been extended to various classes of
operators on some Banach spaces, including weakly compact
operators on $C(K)$ for perfect $K$ and on $L_1(\mu)$ for
atomless $\mu$ (see \cite{Wer0} for an elementary approach). In
the late 1990's the study of the geometry of Banach spaces
having the so-called Daugavet property was initiated. Following
\cite{KSSW0,KSSW}, we say that a Banach space $X$ has the
\emph{Daugavet property} if every rank-one operator $T\in L(X)$
satisfies \eqref{DE}. In such a case, every operator on $X$ not
fixing a copy of $\ell_1$ also satisfies \eqref{DE}; in
particular, this happens to every weakly compact operator on
$X$. This property induces various isomorphic restrictions. For
instance, a Banach space with the Daugavet property does not
have the Radon-Nikod\'{y}m property, it contains $\ell_1$, and it
does not isomorphically embed into a Banach space with
unconditional basis. We refer the reader to the books
\cite{AbrAli-1,AbrAli-2} and the papers \cite{KSSW,WerSur} for
background and more information.

The aim of the cited paper \cite{K-M-M} was to study whether it is
possible to define interesting isometric properties by requiring all
rank-one operators on a Banach space to satisfy a suitable norm
equality. Most of the results gotten there are occlusive, showing
that the most natural attempts to introduce new properties by
considering other norm equalities for operators (like
$\|g(T)\|=f(\|T\|)$ for some functions $f$ and $g$) lead in fact to
the Daugavet property of the space. Nevertheless, there are some
results in the paper valid in the complex case which are not known
to be true for the real case. For instance, it is not known if a
real Banach space $X$ where every rank-one operator $T\in L(X)$
satisfies
$$
\|\Id + T^2\|=1+\|T^2\|
$$
has the Daugavet property (see also \cite{Oik2007} for more
information). Contrary to the Daugavet equation, the equation
above could be satisfied by all bounded linear operators on a
Banach space $X$. Actually, this holds in the simple case
$X=\R$. During an informal discussion on these topics in May
2005, Gilles Godefroy asked the second and the third authors of
this paper about the possibility of finding Banach spaces (of
dimension greater than $1$) having this property (i.e.\ finding
extremely non-complex Banach spaces of dimension greater than
one). Let us comment that, if a Banach space $X$ is extremely
non complex, then it cannot contain a complemented subspace
with complex structure (such as a square) and with summand
$\alpha$-complemented with $\alpha<2$. This can be seen by
applying \eqref{eq:square} to the operator $T\in L(X)$ defined
by $Tx=0$ on the summand, and by $Tx=Jx$ with square of $J$
equal to $-\Id$ on the complemented subspace with complex
structure.

Our (successful) approach to this problem has been to consider
Banach spaces $C(K)$ with few operators in the sense introduced
by the first author of this paper in \cite{Koszmider}. Let us
give two needed definitions.

\begin{definition}\label{def-weak-multiplication-weak-multiplier}
Let $K$ be a compact space and $T\in L(C(K))$. We say that $T$ is a
\emph{weak multiplier} if $T^*=g\Id+S$ where $g:K\longrightarrow \R$
is a function which is integrable with respect to all Radon measures
on $K$ and $S\in W\big(C(K)^*\big)$. If one actually has $T=g\Id+S$
with $g\in C(K)$ and $S\in W\big(C(K)\big)$, we say that $T$ is a
\emph{weak multiplication}.
\end{definition}

In the literature, as far as now, there are several
nonisomorphic types of $C(K)$ spaces with few operators in the
above sense (in ZFC): (1) of \cite{Koszmider} for $K$ totally
disconnected such that $C(K)$ is a subspace of $\ell_\infty$
and  all operators on $C(K)$ are weak multipliers; (2) of
\cite{Koszmider} for  $K$ such that $K\setminus F$ is connected
for every finite $F\subseteq K$, such that $C(K)$ is a subspace
of $\ell_\infty$ and  all operators on $C(K)$ are weak
multipliers; these $C(K)$s, as shown in \cite{Koszmider}, are
indecomposable Banach spaces, hence they are nonisomorphic to
spaces of type (1);  (3) of \cite{Plebanek} for connected $K$
such that all operators on $C(K)$ are weak multiplications;
these spaces are not subspaces of $\ell_\infty$ and hence are
nonisomorphic to spaces of type (1) nor (2) (It is still not
known if such spaces can be subspaces of $\ell_\infty$ without
any special set-theoretic hypotheses; in \cite{Koszmider} it is
shown that the continuum hypothesis is sufficient to obtain
such spaces).

We will show in section~\ref{sec:weakmultiplications} that $C(K)$
spaces on which every operator is a weak multiplication with $K$
perfect give the very first examples of infinite-dimensional
extremely non-complex Banach spaces. The argument is elementary.
With a little bit more of work, we prove in
section~\ref{sec:weakmultipliers} that also $C(K)$ spaces on which
every operator is a weak multiplier with $K$ perfect are extremely
non-complex.

One may think that the fact that a $C(K)$ space is extremely
non-complex implies some kind of fewness of operators. The aim
of section~\ref{sec:further-examples} is to give further
examples to show that this is not the case. We construct
compact spaces $K_1$ and $K_2$ such that $C(K_1)$ and $C(K_2)$
are extremely non-complex, $C(K_1)$ contains a complemented
copy of $C(2^\omega)$ and $C(K_2)$ contains a
($1$-complemented) isometric copy of $\ell_\infty$. One may use
elementary arguments to show that the above $C(K_1)$ and
$C(K_2)$ have many operators however let us quote a deeper
recent result of I.~Schlackow:

\begin{theorem}\cite[Theorem~4.6]{Schl}\label{Schlackow-theorem}
Let $K$ be a compact space. All operators on $C(K)$ are weak multipliers
if and only if the ring $L(C(K))/W(C(K))$ is commutative.
\end{theorem}

When $C({\tilde K})$ is isomorphic to a complemented subspace
of $C(K)$, then the ring $L(C({\tilde K}))/W(C({\tilde K}))$
can be canonically included in $L(C(K))/W(C(K))$. Hence, if a
$C(K)$ has a complemented isomorphic copy of a space $C({\tilde
K})$ then, the corresponding ring is at least ``as much
noncommutative as" that of $C({\tilde K})$. In particular if
$C({\tilde K})$ has operators which are not weak multipliers,
in other words, by the above theorem, the corresponding ring is
noncommutative as in the case of $\ell_\infty$ or
$C(2^\omega)$, then $C(K)$ has operators which are not weak
multipliers.

To obtain $K_1$ and $K_2$ as above using the uniform language
of Boolean algebras and their Stone spaces we will need
$C(K)$'s where all operators are weak multipliers  which are a
bit different than those of types (1)-(3) described above.
Namely we will need $K$ perfect and totally disconnected.
Actually this type was the original construction of earlier
versions of \cite{Koszmider} which later was split into simpler
type (1) which is not perfect and more complex type (2) which
is perfect but not totally disconnected. In
section~\ref{sec:further-examples} we explain how to modify
arguments of \cite{Koszmider} to obtain the desired $K$.

We finish the introduction by commenting that the first attempt
to find extremely non-complex Banach spaces could have been to
check whether the known examples of spaces without complex
structure work. Unfortunately, most of those examples are
reflexive or quasireflexive spaces, and there are no extremely
non-complex spaces in these clases (actually, the unit ball of
an extremely non-complex space does not contain any strongly
exposed point, see \cite{Oik2007}). With respect to the family
of spaces with few operators, it is readily checked that if a
Banach space $X$ has the Daugavet property and every operator
$T\in L(X)$ is a strictly singular perturbation of a multiple
of the identity, then the space is extremely non-complex.
Unfortunately, we have not been able to find any example of
this kind. On the one hand, if a Banach space is hereditarily
indecomposable, then it cannot have the Daugavet property
(which implies containment of $\ell_1$). It is worth mentioning
the existence of a hereditarily indecomposable space whose dual
contains $\ell_1$ \cite[Remark~8.1]{Argyros-Felouzis}. On the
other hand, most of the research about Banach spaces with few
operators deals with the isomorphic structure and does not pay
too much attention to isometric questions.

\section{The first example: weak multiplications} \label{sec:weakmultiplications}
This short section is devoted to give a sufficient condition
for a weak multiplication to satisfy the Daugavet equation,
from which it will be straightforward to get the very first
example of an extremely non-complex Banach space, namely, every
$C(K)$ space where the only bounded linear operators are the
weak multiplications.

\begin{lemma}\label{lemma-weakmultiplication}
Let $K$ be a perfect compact space. If an operator $T\in L(C(K))$
has the form $T=g\Id + S$ where $g\in C(K)$ is non-negative and $S$
is weakly compact, then $T$ satisfies the Daugavet equation.
\end{lemma}

We need the following two old results. The first result goes back to
the 1971 paper by J.~Duncan, C.~McGregor, J.~Pryce and A.~White
\cite[p.~483]{D-Mc-P-W}. The second one was established in the
sixties by I.~Daugavet \cite{Dau} for compact operators on $C[0,1]$
and extended to weakly compact operators by C.~Foia\c s and
I.~Singer and to arbitrary perfect $K$ by A.~Pe\l czy\'nski
\cite[p.~446]{Foias-Singer}. Elementary proofs can be found in
\cite{Wer0}.

\begin{remarks}\label{remark}$ $
\begin{enumerate}
\item[(a)] For every compact space $K$ and every $T\in
    L(C(K))$, one has
$$
\max\{\|\Id+T\|,\ \|\Id-T\|\}=1+\|T\|.
$$
\item[(b)] If $K$ is a perfect compact space, then
    $$
    \|\Id + S\|=1 + \|S\|
    $$
for every $S\in W(C(K))$.
\end{enumerate}
\end{remarks}

\begin{proof}[Proof of Lema~\ref{lemma-weakmultiplication}]
Since the set of those operators on $C(K)$ satisfying the Daugavet
equation is closed and stable by multiplication by positive scalars,
we may suppose that $\underset{t\in K}{\min}\ g(t)>0$ and
$\|g\|\leqslant1$. Now, by using Remark~\ref{remark}.a we have that
$$
\max\left\{\|\Id+g\,\Id+S\|,\ \|\Id-(g\,\Id+S)\|\right\}=1+\|g\,\Id+S\|.
$$
So, we will be done by just proving that
$$
\|\Id-(g\,\Id+S)\|<1+\|g\,\Id+S\|.
$$
On the one hand, it is easy to check that
$$
\|\Id-(g\,\Id+S)\|\leqslant \|\Id-g\,\Id\|+\|S\|=1-\underset{t\in K}{\min}\
g(t)+\|S\|.
$$
On the other hand, we observe that
\begin{align*}
\|g\,\Id+S\|&=\|\Id+S+(g\,\Id-\Id)\|\geqslant\|\Id+S\|-\|g\,\Id-\Id\|\\
&=1+\|S\|-\left(1-\underset{t\in K}{\min}\ g(t)\right)=\|S\|+\underset{t\in
K}{\min}\ g(t)
\end{align*}
where we used Remark~\ref{remark}.b. Since $\underset{t\in K}{\min}\
g(t)>0$, it is clear that
\begin{equation*}
\|\Id-(g\,\Id+S)\|<1+\|g\,\Id+S\|.\qedhere
\end{equation*}
\end{proof}

Suppose now that all the operators on a $C(K)$ space are weak
multiplications and $K$ is perfect. Then, for every $T\in
L(C(K))$ one has $T^2=g\Id+S$ where $g\in C(K)$ is non-negative
and $S$ is weakly compact. The above lemma then yields the
following result.

\begin{theorem}
Let $K$ be a perfect compact space such that every operator on
$C(K)$ is a weak multiplication. Then, $C(K)$ is extremely
non-complex.
\end{theorem}

As we commented in the introduction, there are (even in ZFC) perfect
compact spaces whose operators are weak multiplications
\cite{Plebanek}. Therefore, the above result really gives the
existence of extremely non-complex infinite-dimensional Banach
spaces.

\begin{corollary}\label{cor-exist-extremelync}
There exist infinite-dimensional extremely non-complex
Banach spaces.
\end{corollary}

Let us finish the section showing that our requirement for the
compact space to be perfect is not only methodological.

\begin{remark}
{\slshape Let $K$ be a compact space with at least two points.
If $C(K)$ is extremely non-complex, then $K$ is perfect.\ }
\begin{proof}
If there exists an isolated point $t_0\in K$, we write
$K'=K\setminus \{t_0\}$, we take $t_1\in K'$, and define $T\in
L(C(K))$ by
$$
\bigl[T(f)\bigr]=\bigl(3 f(t_1) - f(t_0)\bigr)\,
\bigl(2\chi_{\{t_0\}} + \chi_{K'}\bigr) \qquad \bigl(f\in C(K)\bigr).
$$
It is readily checked that $T^2=T$, $\|T\|=8$, and $\|\Id + T\| \leq
7 < 1 + \|T\|$.
\end{proof}
\end{remark}

\section{More examples: weak multipliers} \label{sec:weakmultipliers}
Our aim in this section is to enlarge the class of extremely
non-complex Banach spaces by adding those $C(K)$ spaces with perfect
$K$ for which all operators are weak multipliers. Let us first fix
some notation and preliminary results.

Given a compact space $K$, by the Riesz representation theorem, the
dual of the Banach space $C(K)$ is isometric to the space $M(K)$ of
Radon measures on $K$, i.e.\ signed, Borel, scalar-valued, countably
additive and regular measures. More precisely, given $\mu\in M(K)$
and $f\in C(K)$ the duality is given by
$$
\mu(f)=\int f d\mu.
$$
We introduce one more ingredient which will play a crucial role
in our arguments of this section. Given an operator $U\in
L\big(M(K)\big)$, we consider an associated function
$g_U:K\longrightarrow [-\|U\|,\|U\|]$ given by
$$
g_U(x)=U(\delta_x)(\{x\})\qquad (x\in K).
$$
This obviously extends to operators on $C(K)$ by passing to the
adjoint, that is, for $T\in L\big(C(K)\big)$ one can consider
$g_{T^*}:K\longrightarrow [-\|T\|,\|T\|]$. This tool was used in
\cite{Wer0} under the name ``stochastic kernel'' to give an
elementary approach to the Daugavet equation on $C(K)$ spaces. One
of the results in the aforementioned paper, which we state for the
convenience of the reader, will be useful in the remainder of our
exposition.

\begin{lemma}\cite[Lemma~3]{Wer0}\label{Werner-theorem}
Let $K$ be a compact space and $T\in L\big(C(K)\big)$. If the
set $\{x\in K \ : \ g_{T^*}(x)\geq 0\}$ is dense in $K$, then
$T$ satisfies the Daugavet equation.
\end{lemma}

The next result tells us that for weakly compact operators on $M(K)$
the associated functions take ``few'' values.

\begin{lemma}\label{lemma-weak-compact-oper-M(K)}
Let $K$ be a compact space and $U\in W\big(M(K)\big)$. Then,
for every $\varepsilon>0$ the set $\{x\in K \ : \
|g_U(x)|>\varepsilon\}$ is finite.
\end{lemma}

\begin{proof}
If for some $\eps>0$ the set $\{x\in K \ : \ |g_U(x)|>\eps\}$ is
infinite, then there is an infinite sequence $\{x_n\}_{n\in\N}$ of
different points in $K$ satisfying
$$
g_U(x_n)=U(\delta_{x_n})(\{x_n\})>\eps
$$
for every $n\in\N$. By using the regularity of the measures
$\{U(\delta_n)\}_{n\in\N}$ and passing to a subsequence of
$\{x_n\}_{n\in\N}$ if necessary, we can find a family of pairwise
disjoint open sets $\{V_n\}_{n\in\N}$ such that $x_n\in V_n$ and
$$
|U(\delta_{x_n})(V_n\setminus \{x_n\})|<\frac{\eps}{2}
$$
for every $n\in\N$, which implies
$$
|U(\delta_{x_n})(V_n)|>\frac{\eps}{2}\qquad (n\in\N).
$$
This, together with the Dieudonn\'{e}-Grothendieck Theorem (see
\cite[VII.14]{Die}), tells us that the sequence of measures
$\{U(\delta_{x_n})\}_{n\in\N}$ is not relatively weakly compact and,
therefore, that $U$ is not weakly compact.
\end{proof}

As an immediate consequence we get the following corollary.

\begin{corollary}\label{cor-weak-compact-countable}
Let $K$ be a compact space and $U\in W\big(M(K)\big)$. Then,
the set $\{x\in K \ : \ g_{U}(x)\neq0\}$ is at most countable.
\end{corollary}

\begin{theorem}\label{thm-weak-mult-perfect}
Let $K$ be a perfect compact space and $T\in L\big(C(K)\big)$
an operator such that $T^*=g\Id+S$ where $S\in W\big(M(K)\big)$
and $g$ is a Borel function satisfying $g\geq 0$. Then, the set
$$
\{x\in K \ : \ T^*(\delta_x)(\{x\})\geq 0\}
$$
is dense in $K$ and, therefore, $T$ satisfies the Daugavet
equation.
\end{theorem}

\begin{proof} It is clear that for every $x\in K$ one has
$T^*(\delta_x)(\{x\})\geq 0$ provided that $S(\delta_x)(\{x\})=0$
and, therefore,
$$
\{x\in K \ : \ T^*(\delta_x)(\{x\})\geq 0\}\supset\{x\in K \ : \
S(\delta_x)(\{x\})=0\}.
$$
Now, we observe that the last set is dense in $K$ by
Corollary~\ref{cor-weak-compact-countable} and the fact that
nonempty open sets in perfect compact spaces are uncountable (if
$U\subset K$ is open and countable we can find an open set $V$
satisfying $\overline{V}\subset U$ and, therefore, $\overline{V}$ is
a countable compact space and thus scattered by
\cite[Proposition~8.5.7]{Sem}. So $V$ has an isolated point which is also
isolated in $K$ since $V$ is open, contradicting the perfectness of
$K$). Finally, the proof concludes with the use of
Lemma~\ref{Werner-theorem}.
\end{proof}

We can now state and prove the main result of the section.

\begin{theorem}
Let $K$ be a perfect compact space so that every operator on
$C(K)$ is a weak multiplier. Then, $C(K)$ is extremely
non-complex.
\end{theorem}

\begin{proof} Given $T\in L\big(C(K)\big)$, there exist a bounded
Borel function $g$ and $S\in W\big(M(K)\big)$ with
$T^*=g\Id+S$, so
$$
(T^2)^*=(T^*)^2=(g\Id+S)^2=g^2\Id+S'
$$
where $S'$ is weakly compact, and we use the previous theorem.
\end{proof}

As we commented in the introduction, there are infinitely many
nonisomorphic spaces $C(K)$ on which every operator is a weak
multiplier, providing infinitely many nonisomorphic extremely
non-complex Banach spaces.

\begin{corollary}
There exist infinitely many nonisomorphic infinite-dimensional
extremely non-complex Banach spaces.
\end{corollary}

\begin{remark}\label{remark:Schlackow}
It is clear that being extremely non-complex is not an
isomorphic property. This is especially clear for spaces of
continuous functions, since when a compact space $K$ is
infinite, then $C(K)$ is isomorphic to $C(K')$ where $K'$ has
an isolated point ($K'$ is $K$ with two points identified and
one external point added). Anyhow, I.~Schlackow has proved very
recently \cite[Theorem 4.12.]{Schl} that when $K$ and $L$ are
perfect compact spaces, $C(K)$ and $C(L)$ are isomorphic and
every operator on $C(K)$ is a weak multiplier, then so does
every operator on $C(L)$. We will see in the next section (see
Remark~\ref{remark:extremely non-complex}) that the property of
all operators being weak multipliers cannot be replaced by the
property of $C(K)$ being extremely non-complex.
\end{remark}

\section{Further examples}\label{sec:further-examples}
In this section we will construct more examples of extremely
non-complex spaces in the family of $C(K)$ spaces. All the examples
that we have exhibited so far share the characteristic of having few
operators, thus it is natural to ask whether the property of being
an extremely non-complex space is related to the ``lack'' of
operators. This section is devoted to answer this question by
presenting extremely non-complex $C(K)$ spaces with many operators
that are not weak multipliers. To construct such spaces we will
consider a classical chain of inter-related structures: a Boolean
algebra $\mathcal{A}$, its Stone space $K$, the Banach space $C(K)$
and its dual $M(K)$. Let us fix some notation, terminology, and
standard facts related to these structures.

Given a compact space $K$, the clopen subsets of $K$ form a Boolean
algebra which will be denoted by $\Clop(K)$. It is well known that a
compact space is totally disconnected if and only if it is
zero-dimensional i.e., it has a basis of topology consisting of
clopen sets \cite[Theorem~7.5]{Kop}. One can also recover all
totally disconnected compact spaces as the Stone spaces of abstract
Boolean algebras \cite[Theorem~7.10]{Kop}. If $\mathcal{A}$ is a
Boolean algebra the Stone space $K$ of $\mathcal{A}$ is the set of
all ultrafilters of $\mathcal{A}$ endowed with the topology in which
the basic sets are defined as $[A]=\{u\in K\ : \ A\in u \}$ for any
$A\in \mathcal{A}$. By the Stone duality (see
\cite[Theorem~8.2]{Kop}) epimorphisms of Boolean algebras correspond
to monomorphims of their Stone spaces which in turn correspond to
epimorphisms of their Banach spaces of continuous functions. In
particular, we will be interested in homomorphisms of Boolean
algebras which are the identity on their images. For background on
Boolean algebras and Stone spaces we refer the reader to
\cite{Hal,Kop,Sem,Sik}.

\begin{fact}\label{fact-Boolean-projections}
{\slshape Let $P:\Clop(K)\longrightarrow \Clop(K)$ be a homomorphism
of Boolean algebras which is the identity on its image, then the
mapping $\widetilde{P}: C(K)\longrightarrow C(K)$ given by
$$
\widetilde{P}(\chi_A)=\chi_{P(A)} \qquad (A\in\Clop(K))
$$
is a norm-one projection on $C(K)$ (it extends to $C(K)$ by
linearity and the Stone-Weierstrass Theorem) whose image is
isometric to $C(L)$ where $L$ is the Stone space of the Boolean
algebra $P(\Clop(K))$.}
\end{fact}

\begin{proof}
By the Stone-Weierstrass Theorem functions of the form
$\Sigma_{i\leq n} a_i\chi_{A_i}$ span a dense subspace of $C(K)$
where $n\in \N$, $a_i\in \R$ and $A_i\in \Clop(K)$ are pairwise
disjoint. This, together with $P^2=P$ and ${\tilde P}(\Sigma_{i\leq
n} a_i\chi_{A_i}) =\Sigma_{i\leq n} a_i\chi_{P(A_i)}$, implies that
${\tilde P}^2={\tilde P}$. On the other hand, we observe that
$$
\left\|\sum_{i\leq n} a_i\chi_{A_i}\right\|=\sup\{|a_i|: i\leq n, \
A_i\not=\varnothing\}
$$
for pairwise disjoint $A_i\in\Clop(K)$ and that $P$ as a
homomorphism of Boolean algebras preserves the disjointness, which
allow us to deduce that $\|{\tilde P}\|\leq 1$.
\end{proof}

For a totally disconnected compact space $K$, the restriction of a
Radon measure on $K$ to the Boolean algebra $\Clop(K)$ is a finitely
additive signed and bounded measure, that is, $\mu(a\vee
b)=\mu(a)+\mu(b)$ where $a,b\in \Clop(K)$ and $\vee$ denotes the
supremum in $\Clop(K)$. Conversely, any bounded finitely additive
signed measure on such a Boolean algebra extends uniquely to a Radon
measure on Borel subsets of $K$ (see \cite[\S~18.7]{Sem}, for
example). The following remark shows that pointwise convergence of
measures on the Boolean algebra gives weak$^*$ convergence of the
corresponding Radon measures.

\begin{remark}\label{remark-weak*-convergence}
Suppose that $\mu_n$,$\mu$ with $n\in\N$ are uniformly bounded Radon
measures on a totally disconnected compact space $K$ and denote
$\nu=\mu|_{\Clop(K)}$ and $\nu_n={\mu_n}|_{\Clop(K)}$ the associated
finitely additive measures on $\Clop(K)$. Then,
$\{\nu_n(A)\}_{n\in\N}$ converges to $\nu(A)$ for every clopen
subset $A\subseteq K$ if, and only if, $\{\mu_n\}_{n\in\N}$
converges weakly$^*$ to $\mu$.
\begin{proof}
It follows from  the uniform boundedness of the sequence and the
density of the span of the characteristic functions of clopen sets
in $C(K)$ which is an immediate consequence of the Stone-Weierstrass
theorem.
\end{proof}
\end{remark}

Let us present the last elements we need: given a compact space $K$
and $K_i$, $K_j$ clopen subsets of $K$, $K_{-i}$ will stand for the
set $K\setminus K_{i}$ and $P_j: C(K)\longrightarrow C(K_j)$ and
$I_i:C(K_i)\longrightarrow C(K)$ will denote the operators defined
by
$$
 P_j(h)=h|_{K_j} \quad
\big(h\in C(K)\big) \quad \text{and} \quad
\begin{cases} [I_i(f)](x)=f(x) & \text{if} \
x\in K_i \\  [I_i(f)](x)=0 & \text{if} \ x\notin K_i
\end{cases}
\quad \big(f\in C(K_i)\big)
$$
respectively. Moreover, we will also consider the operators $P^i:
M(K)\longrightarrow M(K_i)$ and $I^j: M(K_j)\longrightarrow M(K)$
given by
\begin{align*}
\big[P^i(\mu)\big](L)&=\mu(L) \qquad \big(\mu \in M(K), L\subset
K_i\big)\\ \intertext{and} \big[I^j(\nu)\big](S)&=\nu(S\cap K_j)
\qquad \big(\nu\in M(K_j), S\subset K\big).
\end{align*}
Finally, all the operators above will be also used with index $-i$.
The relationship between these operators is established in the
following easy lemma.

\begin{lemma}\label{lemma-duality-proy-injec}
Let $K$ be a compact space, let $K_i$, $K_j$ be clopen subsets of
$K$, and consider the operators $P_j$ and $I_i$ defined above. Then,
${I_i}^*=P^i$ and ${P_j}^*=I^j$. Moreover, given $g$ a Borel
function on $K$, the following holds:
$$
P^i (g\Id_{M(K)}) I^j=\begin{cases}0 & \text{if\ } K_i\cap
K_j=\varnothing\\g|_{K_i} \Id_{M(K_i)} & \text{if\ } K_i=K_j.
\end{cases}
$$
\end{lemma}

\begin{proof}
For $\mu\in M(K)$ and $f\in C(K_i)$ we have
$$
I^*_i(\mu)(f)=\int_K I_i(f) d\mu=\int_{K_i}fd \big(P^i(\mu)\big)
$$
which gives $I^*_i(\mu)=P^{i}(\mu)$. On the other hand, for $\nu\in
M(K_j)$ and $f\in C(K)$ we can write
$$
P^*_j(\nu)(f)=\int_{K_j} P_j(f)d\nu=\int_{K_j} (f|_{K_j}) d\nu=
\int_K fd\big(I^{j}(\nu)\big)
$$
and, therefore, $P^*_j(\nu)=I^j(\nu)$. Finally, if $K_i\cap
K_j=\varnothing$, given $\nu \in M(K_j)$ and $L\subset K_i$, we
observe that
$$\big[(g \Id_{M(K)})
I^j\big](\nu)(L)= \int_K g\chi_L dI^j(\nu)$$ so, since $I_j(\nu)$ is
a measure which assumes value zero on all sets disjoint from $K_j$,
the same is true for $g \Id_{M(K)} I^j(\nu)$, and so $\big[P^i (g
\Id) I^j\big](\nu)=0$.

If $K_i=K_j$, given $\nu \in M(K_i)$ and $L\subset K_i$, we observe
that
\begin{align*}
\big[P^i (g \Id)
I^i\big](\nu)(L)&=P^i\big[gI^i(\nu)\big](L)\\
&=\big[(g|_{K_i})I^j(\nu)\big](L)=\int g|_{K_i}\chi_L d\nu=
g|_{K_i}\Id_{M(K_i)}(\nu)(L)
\end{align*}
from which it is immediate to deduce the moreover part.
\end{proof}

In the next two results we prove the existence of a family of $C(K)$
spaces which are Grothendieck and so that there are few operators
between any pair of them. We recall that a Banach space $X$ is said
to be Grothendieck if every weak$^*$ convergent sequence in $X^*$ is
weak convergent and the fact that a $C(K)$ space in which every
operator is a weak multiplier is Grothendieck
\cite[Theorem~2.4]{Koszmider}.

\begin{prop}\label{prop-perfect-discon-few}
There is a compact infinite totally disconnected and perfect space
$K$ such that all operators on $C(K)$ are weak multipliers.
\end{prop}
\begin{proof}
We will describe a modification of the construction from Section~$3$
of \cite{Koszmider}. As seen in Lemma~$3.2$ of \cite{Koszmider}, the
only properties of points $n^*$ for $n\in\N$ of the constructed $K$
which are needed to prove that every operator on $C(K)$ is a weak
multiplier are those stated in Lemma~$3.1$ and the density in $K$.
Thus, to prove the proposition it is  enough to construct an
atomless Boolean algebra ${\mathcal A} \subseteq\wp(\N)$ (the lack
of atoms is equivalent to the fact that the Stone space $K$ is
perfect) and a countable dense subset $\{q_n:n\in\N\}$ of its Stone
space such that given

\begin{itemize}
\item [a)] a sequence $(A_n:n\in \N)$ of pairwise disjoint elements of
${\mathcal A}$,

\item [b)] a sequence $({\ell_n}:n\in \N)$ of distinct natural numbers such
that $q_{\ell_n}\not \in A_m$ for $n,m\in \N$,
\end{itemize}

there is an infinite  $ b\subseteq \N$ such that
\begin{itemize}
\item [c)] $\{ A_m: m\in b\}$ has its supremum $A$ in ${\mathcal A}$
 and
\item [d)] the intersection of the sets
  ${\overline{\{q_{\ell_n}:n\in b\}}}$ and
  ${\overline{\{q_{\ell_n}: n\not\in b\}}}$
in the Stone space $K$ of ${\mathcal A}$
is nonempty.
\end{itemize}

Such an algebra is  constructed as in the proof of Lemma~$3.1$ of
\cite{Koszmider}. Indeed, we give an outline of the modification of
that proof (observe that the notation used in \cite{Koszmider} will
be kept). The only complication is that the ultrafilters $q_n$ are
not absolutely defined as the ultrafilters $n^*$. This means that at
each inductive step of the construction of the subalgebras
${\mathcal A}_\alpha$ for $\alpha\leq 2^\omega$ one needs to extend
the ultrafilter $q_n|\alpha$ of ${\mathcal A}_\alpha$ to an
ultrafilter $q_n|(\alpha+1)$ of the new bigger algebra ${\mathcal
A}_{\alpha+1}$, since at the limit stages the ultrafilters are
determined by their intersections with the previous algebras.

This problem has been encountered in the connected construction of
$5.1$ of \cite{Koszmider} and is resolved in the same way.  Namely,
if $q_n|\alpha$ has only one extension to an ultrafilter of
${\mathcal A}_{\alpha+1}$, then one puts it as
 $q_n|(\alpha+1)$, and otherwise one needs to make some uniform choice, for
example $q_n|(\alpha+1)$ is such an ultrafilter of ${\mathcal
A}_{\alpha+1}$ which extends $q_n|\alpha$  and does not contain
$A_\alpha$, the generator of ${\mathcal A}_{\alpha+1}$ over
${\mathcal A}_\alpha$ .

Now, at stage $\alpha<2^\omega$ we are given $\alpha$ premises
of the form
\begin{equation}\label{eq-ss}
{\overline{\{q_n|\alpha: n\in b_\beta\}}}
\cap{\overline{\{q_n|\alpha: n\in
 a_\beta-b_\beta\}}}\not=\emptyset,
\end{equation}
for $\beta<\alpha$ where the closures are taken in the Stone
space of the algebra ${\mathcal A}_\alpha$. We are also given
an antichain $(A_n: n\in \N)$ in the algebra ${\mathcal
A}_\alpha$ and need to preserve the premisses when passing to
the algebra ${\mathcal A}_{\alpha+1}$ generated over ${\mathcal
A}_\alpha$ by an element $A_\alpha$ which is an infinite sum
(in $\wp(\N)$) of some infinite subsequence of $(A_n: n\in b)$
where $b$ should be an arbitrary infinite subset of some
$a\subseteq\N$. So we need to make a good choice of $a\subseteq
\N$.

We may assume that all points of the intersection of the
closures from \eqref{eq-ss} are from outside of the clopen sets
$[A_n]$, since otherwise the premises are always preserved for
any choice of the subsequence. Now we may find an infinite
$a\subseteq \N$ such that that
\begin{equation}\label{eq-sss}
{\overline{\{q_n|\alpha\not\in\bigcup_{m\in a}[A_m] : n\in
b_\beta, \}}} \cap{\overline{\{ q_n|\alpha\not\in\bigcup_{m\in
a}[A_m] :  n\in
 a_\beta-b_\beta,\}}}\not=\emptyset,
\end{equation}
holds for every $\beta<\alpha$.
 Thus, if the extension is obtained from any infinite $b\subseteq a$, the points
$q_n|\alpha\not\in\bigcup_{m\in b}[A_m] $ all extend to
ultrafilters which do not contain $A_\alpha$ hence the
preservation of the premises \eqref{eq-ss} can be easily proved
and so the inductive construction can continue.

The method of finding $a$ is already employed in $3.1$ of
\cite{Koszmider}. One considers an almost disjoint family
$\{a^\theta: \theta<2^\omega\}$ of size $2^\omega$ of infinite
subsets of $\N$, one chooses $x_\beta$ from the intersections
\eqref{eq-ss} and sees that for at most one choice of $\theta$
the point $x_\beta$ does not belong to
${\overline{\{q_n|\alpha\not\in\bigcup_{m\in a^\theta}[A_m] :
n\in b_\beta,  \}}}$, the same holds for the other part of the
premise. Thus, by the counting argument we have  that there is
$\theta<2^\omega$ such that \eqref{eq-sss} holds for
$a=a^\theta$ for all $\beta<\alpha$.

One is left with checking that the extensions of Boolean
algebras which are used in $3.1$ of \cite{Koszmider} do not
introduce atoms, but it is clear as we extend by adding an
infinite union of elements of the previous algebra. Other
arguments are as in $3.1$ of \cite{Koszmider}.
\end{proof}

\begin{prop}\label{prop-family-Ki-weak-multipliers}
There is a family $(K_i)_{i\in \N}$ of pairwise disjoint perfect and
totally disconnected compact spaces such that every operator on
$C(K_i)$ is a weak multiplier (thus, $C(K_i)$ is Grothendieck) and
for $i\not=j$ every operator $T\in L\big(C(K_i),C(K_j)\big)$ is
weakly compact.
\end{prop}

\begin{proof}
Consider $K$ perfect and totally disconnected so that every operator
in $C(K)$ is a weak multiplier, fix a family $(K_i)_{i\in \N}$ of
pairwise disjoint clopen subsets of $K$, and let us prove that this
family satisfies the desired conditions. Fixed $i,j \in \N$ and an
operator $T: C(K_i)\longrightarrow C(K_j)$, we define
$\widetilde{T}:C(K)\longrightarrow C(K)$ by
$$
\widetilde{T}=I_j T P_i
$$
which is a weak multiplier by hypothesis and so there are $g$ a
Borel function on $K$ and $S\in W\big(M(K)\big)$ so that
$\widetilde{T}^*=g\Id+S$. Besides, it is clear that $P_\ell I_\ell$
is the identity on $C(K_\ell)$ for every $\ell\in \N$, so we have
that $T=P_j \widetilde{T} I_i$ and, therefore,
$$
T^*=I_i^* \widetilde{T}^* P_j^*=I_i^*(g\Id+S) P_j^*= I_i^* g\Id
 P_j^*+ I_i^* S  P_j^*.
$$
Finally, the operator $I_i^* S  P_j^*$ is weakly compact and
Lemma~\ref{lemma-duality-proy-injec} tells us that
$$
I_i^* g\Id  P_j^*=\begin{cases} 0 & \text{for\ } i\not=j\\
g|_{K_i}\Id_{M(K_i)} & \text{for\ } i=j,
\end{cases}
$$
finishing thus the proof.
\end{proof}

In the following we will be considering some compactifications of
disjoint unions of perfect compact spaces $K_i$, that is, compact
spaces where $\bigcup_{i\in\N}K_i$ is open and dense. The next
result, which will be the cornerstone of our further discussion,
gives a sufficient condition on such compactifications for obtaining
that the associated space of continuous functions is extremely
non-complex.

\begin{prop}\label{prop-suf-condition}
Let $(K_i)_{i\in\N}$ be the family given in
Proposition~\ref{prop-family-Ki-weak-multipliers} and let $K$ be a
compactification of ${\bigcup}_{i\in\N}K_i$ so that every operator
$T:C(K_i)\longrightarrow C(K_{-i})$ is weakly compact or every
operator $T:C(K_{-i})\longrightarrow C(K_{i})$ is weakly compact.
Then, $C(K)$ is an extremely non-complex space.
\end{prop}

\begin{proof}
Fixed $T\in L\big(C(K)\big)$ and $i\in \N$, we can write
$$
P_iT^2I_i=P_iT(I_iP_i+I_{-i}P_{-i})TI_i= (P_iTI_i)(P_iTI_i)+
(P_iTI_{-i})(P_{-i}TI_i)
$$
and, therefore,
$$
(P_iT^2I_i)^*=(P_iTI_i)^*(P_iTI_i)^*+\big[(P_iTI_{-i})(P_{-i}TI_i)\big]^*.
$$
We observe that the second summand is weakly compact by hypothesis
since $P_{-i}TI_i\in L\big(C(K_i),C(K_{-i})\big)$ and $P_iTI_{-i}\in
L\big(C(K_{-i}),C(K_{i})\big)$. Besides, $P_iTI_i$ is an operator on
$C(K_i)$ thus, there exist a bounded Borel function $g$ and a weakly
compact operator $S\in L\big(M(K)\big)$ so that $(P_iTI_i)^*=g\Id+S$
and hence we can deduce that
$$
(P_iT^2I_i)^*=(g\Id+S)^2 + \big[(P_iTI_{-i})(P_{-i}TI_i)\big]^*=
g^2\Id+S'
$$
where $S'$ is a weakly compact operator on $M(K)$. Now, since $K_i$
is perfect for every $i\in \N$, we can use
Theorem~\ref{thm-weak-mult-perfect} to get that the set
$$
\{x\in K_i \ : \ \big(P_i T^2 I_i\big)^*(\delta_x)(\{x\})\geq0\}
$$
is dense in $K_i$. Finally, we use that for $x\in K_i$ one has
$$
\big(P_i T^2 I_i\big)^*(\delta_x)(\{x\})=\big(P^i (T^2)^*
I^i\big)(\delta_x)(\{x\})=(T^2)^*(\delta_x)(\{x\})
$$
and the fact that $\underset{i\in\N}{\bigcup} K_i$ is dense in $K$
to deduce that the set
$$
\{x\in K \ : \ (T^2)^*(\delta_x)(\{x\})\geq0 \}
$$
is dense in $K$ which, by making use of Lemma~\ref{Werner-theorem},
tells us that $T^2$ satisfies the Daugavet equation.
\end{proof}

Our next aim is to construct compact spaces $K$ in such a way that
$C(K)$ is an extremely non-complex space and so that there exist
operators on $C(K)$ which are not weak multipliers. To do so, we
consider a suitable family of totally disconnected compact spaces
$(K_i)_{i\in\N}$ and we obtain our compact spaces as zero
dimensional compactifications of the disjoint union
$\bigcup_{i\in\N} K_i$. This kind of spaces has been completely
described (see \cite[Proposition~8.8]{Kop}, for instance), namely,
they are the Stone spaces of the Boolean subalgebras of the
cartesian product $\mathcal{A}=\Pi_{i\in\N} \Clop(K_i)$ which
contain the subalgebra given by
$$
\Pi_{i\in\N}^w \Clop(K_i)=\Big\{a\in \mathcal{A} \ : \ \{i\in\N \ :
\ a_i\neq\varnothing \} \text{ is finite or } \{i\in\N \ : \ a_i\neq
K_i\} \text{ is finite } \Big\}.
$$
Therefore, we will be interested in constructing such a type of
Boolean algebras. Indeed, let $\mathcal{B}\subset \wp(\N)$ be a
Boolean algebra containing all finite and cofinite subsets of $\N$
then, $\bigoplus^\mathcal{B}_{i<\omega} \Clop(K_i)$ will denote the
Boolean algebra isomorphic to the algebra of subsets of
$\bigcup_{i\in\N} K_i$ consisting of elements of the form
$$
c=c(b,F,\{a_j \ : \ j\in F\}) =\bigcup_{i\in b}K_i\cup\{a_j \ : j\in
F\}
$$
where $b\in\mathcal{B},\ F$ is a finite subset of $\N$ and $a_j\in
\Clop(K_j)$ for all $j\in F$. By the preceding observations, the
Stone space $K$ of $\bigoplus^\mathcal{B}_{i<\omega} \Clop(K_i)$ is
a compactification of the disjoint union $\bigcup_{i\in\N} K_i$. We
are ready to state and prove a result which includes great part of
the difficulties of our first construction.

\begin{theorem}\label{thm-Grothendieck-spaces}
Let $\mathcal{B}\subset \wp(\N)$ be a Boolean algebra containing all
finite and cofinite subsets of $\N$, $(K_i)_{i\in\N}$ a family of
totally disconnected compact spaces, and $K$ the Stone space of
$\bigoplus^\mathcal{B}_{i<\omega} \Clop(K_i)$. Suppose that $C(K_i)$
are Grothendieck Banach spaces so that, for $j\neq i$, every
operator in $L\big(C(K_i),C(K_j)\big)$ is weakly compact and suppose
that $B_{M(L)}$ is weak$^*$-sequentially compact where $L$ is the
Stone space of $\mathcal{B}$. Then, every operator from $C(K_{i})$
into $C(K_{-i})$ is weakly compact.
\end{theorem}

\begin{proof}
Suppose that there are $i\in \N$ and a bounded operator $T:
C(K_{i})\rightarrow C(K_{-i})$ which is not weakly compact. Then by
Gantmacher's Theorem its adjoint neither is weakly compact, which
means by the Dieudonn\'{e}-Grothendieck Theorem \cite[VII.14]{Die} that
there are a bounded sequence of measures $\mu_n\in M(K_{-i})$,
pairwise disjoint clopen subsets $A_n$ of $K_{i}$, and
$\varepsilon>0$ such that
$$
|T^*(\mu_n)(A_n)|>\varepsilon \qquad (n\in\N).
$$
For each $j\neq i$, we use Lemma~\ref{lemma-duality-proy-injec} to
write $(P_j I_{-i} T)^*=T^*P^{-i} I^j$ and we observe that $P_j
I_{-i}T$ is an operator from $C(K_{i})$ into $C(K_{j})$ and thus it
is weakly compact by hypothesis. Hence, $\big\{T^*P^{-i}
I^j(P^jI^{-i}(\mu_n))\big\}_{n\in\N}$ is relatively weakly compact
and so
$$
\big\{T^*P^{-i}
I^j(P^jI^{-i}(\mu_n))(A_n)\big\}_{n\in\N}\longrightarrow0
$$
for every $j\neq i$. Therefore, by a diagonalization process we may
find subsequences of $\{\mu_n\}_{n\in\N}$ and $\{A_n\}_{n\in\N}$
(which we also call $\{\mu_n\}_{n\in\N}$ and $\{A_n\}_{n\in\N}$) so
that
$$
\big|T^*P^{-i} I^j(P^jI^{-i}(\mu_n))(A_n)\big|<\frac{\eps}{2^{j+1}}
$$
for every $n\in\N$ and $j\in\{1,\dots,n\}\setminus\{i\}$. Now, we
consider the family of measures on $K_{-i}$ given by
$$
\widetilde{\mu}_n=\mu_n-\sum_{1\leq j\leq n}^{j\neq i} P^{-i} I^j
P^jI^{-i}(\mu_n)
$$ which is bounded since the measures involved have disjoint
supports (in fact $|\widetilde{\mu}_n|\leq 2|\mu_n|$) and satisfies
that
$$
|\widetilde{\mu}_n|(K_j)=0 \qquad
\big(j\in\{1,\dots,n\}\setminus\{i\}\big)\qquad \text{and} \qquad
|T^*(\widetilde{\mu}_n)(A_n)|>\frac{\eps}{2}
$$
for every $n\in \N$. Indeed, we can estimate as follows:
\begin{align*}
|T^*(\widetilde{\mu}_n)(A_n)|&\geq |T^*(\mu_n)(A_n)|-\sum_{1\leq
j\leq n}^{j\neq i}\big|T^*P^{-i} I^j(P^jI^{-i}(\mu_n))(A_n)\big|
>\eps -\frac{\eps}{2}
\end{align*}
which tells us that the sequence
$\{T^*(\widetilde{\mu}_n)\}_{n\in\N}$ is not relatively weakly
compact.

Besides, let $\mathcal{B}'$ be the subalgebra of $\Clop(K)$
corresponding to the elements of the form $\bigcup_{i\in b}K_i$ for
$b\in \mathcal{B}$, which is clearly isomorphic to $\mathcal{B}$.
For $n\in\N$, we consider the finitely additive measure on
$\mathcal{B}'$ given by
$\nu_n=I^{-i}(\widetilde{\mu}_n)|_{\mathcal{B}'}$. Since $B_{M(L)}$
is weak$^*$-sequentially compact, we may and do assume without loss
of generality that $\{\nu_n\}_{n\in\N}$ is pointwise convergent to a
finitely additive measure $\nu$ on $\mathcal{B}'$ which, in
addition, satisfies that $\nu(K_j)=0$ for every $j\in \N$. Indeed,
for $j\in\N$ and $n>j$, we can write
$$
I^{-i}(\widetilde{\mu}_n)(K_j)=\widetilde{\mu}_n(K_j\cap K_{-i})=0
$$
and, therefore, $\nu(K_j)=\lim I^{-i}(\widetilde{\mu}_n)(K_j)=0$.
Next, we extend $\nu$ to a finitely additive measure $\mu$ on
$\bigoplus^\mathcal{B}_{i<\omega}\Clop(K_i)$ by putting
$$
\mu(c(b,F,\{a_j \ : \ j\in F\}))= \nu(c(b,\varnothing,\varnothing))
$$
for $c(b,F,\{a_j \ : \ j\in F\})\in
\bigoplus^\mathcal{B}_{i<\omega}\Clop(K_i)$ and we observe that
\begin{align*}
I^{-i}(\widetilde{\mu}_n)\big(c(b,F,\{a_j \ : \ j\in
F\})\big)&=I^{-i}(\widetilde{\mu}_n)\left(\Big(\underset{\ell\in
b}{\textstyle{\bigcup}}K_\ell\Big)\cup \Big(\underset{j\in
F\setminus b}{\textstyle{\bigcup}} a_j\Big)
\right)\\&=\nu_n\Big(\underset{\ell\in
b}{\textstyle{\bigcup}}K_\ell\Big)+
\widetilde{\mu}_n\Big(\underset{j\in F\setminus
(b\cup\{i\})}{\textstyle{\bigcup}}
a_j\Big)\\&=\nu_n\Big(\underset{\ell\in
b}{\textstyle{\bigcup}}K_\ell\Big)
\end{align*}
where the last equality holds for every sufficiently large $n$ since
$|\widetilde{\mu}_n|(K_j)=0$ for every
$j\in\{1,\dots,n\}\setminus\{i\}$. So, we have that
$$
\left\{I^{-i}(\widetilde{\mu}_n)\big(c(b,F,\{a_j \ : \ j\in
F\})\big)\right\}_{n\in\N}\longrightarrow \mu\big(c(b,F,\{a_j \ : \
j\in F\})\big)
$$
which, together with Remark~\ref{remark-weak*-convergence},
tells us that $\{I^{-i}(\widetilde{\mu}_n)\}_{n\in\N}$
converges in the weak$^*$ topology to the unique extension of
$\mu$ to an element of $M(K)$ that we also denote by $\mu$.
Now, since $P^{-i}=(I_{-i})^*$ is weak$^*$-weak$^*$ continuous
and $\widetilde{\mu}_n=P^{-i} I^{-i}(\widetilde{\mu}_n)$ for
$n\in\N$, we obtain that $\{\widetilde{\mu}_n\}_{n\in\N}$
weak$^*$ converges to $P^{-i}(\mu)$ and, therefore,
$\{T^*(\widetilde{\mu}_n)\}_{n\in\N}$ is weak$^*$ convergent.
Finally, the hypothesis of $C(K_i)$ being a Grothendieck space
tells us that $\{T^*(\widetilde{\mu}_n)\}_{n\in\N}$ converges
weakly, contradicting the fact that it is not relatively weakly
compact and completing thus the proof of the theorem.
\end{proof}

\begin{theorem}\label{thm-containing-C-Delta}
There is a compact space $K$ so that $C(K)$ is extremely non-complex
and contains a complemented isomorphic copy of $C(2^\omega)$.
\end{theorem}

\begin{proof}
Let us first recall that a countable independent family in a Boolean
algebra is a family $\{a_n: n\in \N\}$ such that
$$\eps_1a_{n_1}\cap....\cap \varepsilon_k a_{n_k}\not=\varnothing$$
for any distinct choice of $n_1,...n_k$, $k\in \N$ and
$\eps=\pm 1$ where $-a$ denotes the complement of $a$. We
consider an independent family of subsets of $\N$ and the
Boolean algebra $\mathcal{B}$ generated by it and the finite
subsets of $\N$, we take $(K_i)_{i\in\N}$ the family of perfect
and totally disconnected compact spaces given by
Proposition~\ref{prop-family-Ki-weak-multipliers} and we define
$K$ as the Stone space of the Boolean algebra
$\bigoplus^\mathcal{B}_{i<\omega}\Clop(K_i)$. Let us check that
$K$ satisfies the desired conditions: since $\mathcal{B}$ is
countable, its Stone space $L$ has a countable basis of
topology and so it is metrizable, which implies that
$(B_{M(L)}, w^*)$ is metrizable and, therefore, sequentially
compact. Henceforth, since every $C(K_i)$ is Grothendieck, we
can use Theorem~\ref{thm-Grothendieck-spaces} and
Proposition~\ref{prop-suf-condition} to get that $C(K)$ is
extremely non-complex. In order to prove that $C(K)$ has a
complemented copy of $C(2^\omega)$, we fix $x_i\in K_i$ for
$i\in\N$, we consider $\mathcal{B}'$ the Boolean subalgebra of
$\bigoplus^\mathcal{B}_{i<\omega}\Clop(K_i)$ formed by elements
of the form $\bigcup_{i\in b}K_i$ for $b\in \mathcal{B}$ (which
is obviously isomorphic to $\mathcal{B}$), and we define a
projection
$P:\bigoplus^\mathcal{B}_{i<\omega}\Clop(K_i)\longrightarrow\mathcal{B}'$
by
$$
P\big(c(b,F,\{a_j \ : \ j\in F\})\big)=\bigcup_{i\in b}\{K_i:x_i\in
c(b,F,\{a_j \ : \ j\in F\})\}
$$
for $c(b,F,\{a_j\ : \ j\in F\})\in
\bigoplus^\mathcal{B}_{i<\omega}\Clop(K_i)$. As we noted in
Fact~\ref{fact-Boolean-projections}, $P$ induces a norm-one
projection from $C(K)$ onto a subspace isometric to $C(L)$. Finally,
since $\mathcal{B}$ contains an infinite independent family, its Stone's space
maps onto $2^\omega$ and hence $L$ is
an uncountable metric compact space and, therefore, $C(L)$ is
isomorphic to $C(2^\omega)$ by Miljutin's Theorem (see
\cite[Theorem~4.4.8]{Albiac-Kalton}, for instance).
\end{proof}

\begin{remark}
By Theorem~\ref{Schlackow-theorem} and the
comments below it the above space has
many operators which are not weak multipliers.
\end{remark}

We recall that being extremely non-complex is a property stable
under isomorphisms in the class of $C(K)$ spaces where $K$ is a
perfect compact space so that every operator on $C(K)$ is a weak
multiplier (see Remark~\ref{remark:Schlackow}). The next remark
shows that this is no longer true when one leaves this class even if
one keeps perfectness.

\begin{remark}\label{remark:extremely non-complex}
{\slshape There are perfect compact spaces $K$ and $L$ so that
$C(K)$ is isomorphic to $C(L)$, $C(K)$ is extremely non-complex, and
$C(L)$ fails to be extremely non-complex.} Indeed, let $C(K)$ be
from Theorem~\ref{thm-containing-C-Delta} and recall that the proof
of that result gives the existence of a subspace $X$ of $C(K)$ so
that
$$
C(K)\sim X\oplus C(2^\omega)\sim X\oplus C(2^\omega)\oplus
C(2^\omega)\sim C(K)\oplus C(2^\omega)\sim C(K)\oplus_\infty
C(2^\omega) \sim  C(K{\dot \cup} 2^\omega).
$$
The latter space contains a complemented subspace isometric to a
square so that the projection on its complement is of norm one.
Therefore, Remark~4.12 of \cite{K-M-M} tells us that $C(K{\dot \cup}
2^\omega)$ is not extremely non-complex.
\end{remark}

Our next goal is to construct a compact space $K$ so that $C(K)$ is
extremely non-complex and contains $\ell_{\infty}$. In order to use
our machinery we have to define a suitable Boolean algebra: let
$(\mathcal{A}_i)_{i\in\N}$ be a family of Boolean algebras so that
$\mathcal{A}_i$ is isomorphic to an algebra of subsets of $X_i$ for
$i\in\N$, where $(X_i)_{i\in\N}$ is a family of pairwise disjoint
sets. Then, $\bigotimes_{i<\omega}\mathcal{A}_i$ will denote the
Boolean algebra isomorphic to the algebra of subsets of
$\bigcup_{i\in\N}X_i$ consisting of elements
$A\in\wp\Big(\bigcup_{i\in\N}X_i\Big)$ satisfying the condition
$A\cap X_i\in \mathcal{A}_i$ for every $i\in \N$.

\begin{theorem}\label{thm-l-infinity}
Let $(K_i)_{i\in\N}$ be a family of totally disconnected
compact spaces so that $C(K_i)$ does not include any copy of
$\ell_\infty$ and such that every operator from $C(K_i)$ into
$C(K_j)$ is weakly compact for $j\neq i$, and let $K$ be the
Stone space of $\bigotimes_{i<\omega}\Clop(K_i)$. Then, every
operator from $C(K_{-i})$ into $C(K_{i})$ is weakly compact.
\end{theorem}

\begin{proof}
Suppose that there are $i\in \N$ and a bounded operator $T:
C(K_{-i})\rightarrow C(K_{i})$ which is not weakly compact. Then by
Gantmacher's Theorem its adjoint neither is weakly compact, which
means by the Dieudonn\'{e}-Grothendieck Theorem \cite[VII.14]{Die} that
there are a bounded sequence of measures $\mu_n\in M(K_{i})$,
pairwise disjoint clopen subsets $A_n$ of $K_{-i}$, and
$\varepsilon>0$ such that
$$
|T^*(\mu_n)(A_n)|>\varepsilon \qquad (n\in\N).
$$
For each $j\neq i$, we use Lemma~\ref{lemma-duality-proy-injec} to
write $(TP_{-i} I_{j})^*=P^jI^{-i}T^*$ and we observe that $TP_{-i}
I_{j}$ is an operator from $C(K_{j})$ into $C(K_{i})$ and thus it is
weakly compact by hypothesis. Hence,
$\big\{P^jI^{-i}T^*(\mu_n)\big\}_{n\in\N}$ is relatively weakly
compact and so
$$
\big\{P^jI^{-i}T^*(\mu_n)(A_n\cap
K_j)\big\}_{n\in\N}\longrightarrow0
$$
for every $j\neq i$, a fact which is obviously true for $j=i$. We
also observe that for $j\neq i$ and $n\in \N$ we have
\begin{align*}
P^jI^{-i}T^*(\mu_n)(A_n\cap K_j)&=I^{-i}T^*(\mu_n)(A_n\cap
K_j)\\&=T^*(\mu_n)(A_n\cap K_j\cap K_{-i})=T^*(\mu_n)(A_n\cap K_j)
\end{align*}
so, by passing to a convenient subsequence, we get that
$$
|T^*(\mu_n)(A_n\cap K_j)|<\frac{\eps}{2^{j+1}}
$$
for every $j\in \{1,\dots,n\}$. Now, for $n\in\N$, we consider the
clopen subset of $K_{-i}$ given by
$\widetilde{A}_n=A_n\setminus\overset{n}{\underset{j=1}{\cup}}K_j$
which satisfies
\begin{align}\label{eq:condition-non-w-compact}
|T^*(\mu_n)(\widetilde{A}_n)|&=\left|T^*(\mu_n)(A_n)-
T^*(\mu_n)\Big(A_n\cap\Big(\overset{n}{\underset{j=1}{\cup}}K_j\Big)\Big)\right|\\
&\geq|T^*(\mu_n)(A_n)|-\sum_{j=1}^n|T^*(\mu_n)(A_n\cap
K_j)|>\eps-\frac{\eps}{2}\,.\notag
\end{align}
Besides, for $j\in\N\setminus\{i\}$ we define $\mathcal{A}_j'$ as
the finite subalgebra of $\Clop(K_j)$ generated by
$$
\Big\{\widetilde{A}_n\cap K_j \ : \ n<j\Big\}
$$
and we observe that $\widetilde{A}_n$ belongs to
$\bigotimes_{j\not=i}\mathcal{A}_j'$ for every $n\in\N$.
Moreover, it is not hard to check that
$\bigotimes_{j\not=i}\mathcal{A}_j'$ is a subalgebra of
$\bigotimes_{j\not=i}\Clop(K_j)$ isomorphic to $\wp(\N)$, so
the corresponding subspace $Y$ of $C(K_{-i})$ is isomorphic to
$\ell_\infty$. In fact, $Y$ is the closure of the space spanned
by the characteristic functions of clopen sets in the image of
the projection
$P:\bigotimes_{j\not=i}\Clop(K_j)\longrightarrow\bigotimes_{j\not=i}\mathcal{A}_j'$
given by
$$
P(A)=\bigcup \{a_{j,k} \ : \ x_{j,k}\in A\} \qquad
\left(A\in\textstyle{\bigotimes}_{j\not=i}\Clop(K_j)\right)
$$
where $\{a_{j,k}: k<k_j\}$ denotes the collection of all atoms of
$\mathcal{A}_j'$ for all $j\in\N\setminus\{i\}$ and some $k_j\in\N$,
and $x_{j,k}\in a_{j,k}$ are some fixed points  for
$j\in\N\setminus\{i\}$ and  $k<k_j$. Since every element of
$\bigotimes_{j\not=i}\mathcal{A}_j'$ is contained in $K_{-i}$ we get
that $Y\subset C(K_{-i})$. Now, for $n\in\N$, we observe that
$$
\left(T|_{Y}\right)^*(\mu_n)(\widetilde{A}_n)=T^*(\mu_n)(\widetilde{A}_n)
$$
therefore, by using \eqref{eq:condition-non-w-compact}, we can
deduce that $T|_{Y}$ is an operator from an injective space of
continuous functions into $C(K_i)$ which is not weakly compact
and so it fixes a copy of $c_0$ (see \cite[Section~4]{Pel} or
\cite[Theorem~4.5]{Ros2}). Therefore, we can use
\cite[Proposition~1.2]{Ros} to deduce that $T|_{Y}$ fixes a
copy of $\ell_\infty$ since $Y$ is injective. This gives a
contradiction since $C(K_i)$ does not contain any copy of
$\ell_\infty$ by the hypothesis.
\end{proof}

\begin{theorem}
There is a compact space $K$ so that $C(K)$ is extremely
non-complex and contains an isometric complemented copy of
$\ell_\infty$.
\end{theorem}

\begin{proof}
Let $(K_i)_{i\in\N}$ be the family of perfect and totally
disconnected compact spaces given in
Proposition~\ref{prop-family-Ki-weak-multipliers} then, the Stone
space $K$ of $\bigotimes_{i<\omega}\Clop(K_i)$ satisfies the
requested conditions. Indeed, we choose $x_i\in K_i$ for every
$i\in\N$, we consider the subalgebra $\mathcal{B}'$ of $\Clop(K)$
consisting of the elements of the form $\bigcup_{i\in b}K_i$ for
$b\subseteq \N$, and we define a projection
$P:\bigotimes_{i<\omega}\Clop(K_i)\longrightarrow\mathcal{B}'$ by
$$
P(A)=\bigcup_{i\in b_A} K_i\qquad
\Big(A\in\textstyle{\bigotimes}_{i<\omega}\Clop(K_i)\Big)
$$
where $b_A=\{i\in\N \ : \ x_i\in A\}$. By
Fact~\ref{fact-Boolean-projections}, $P$ induces a norm-one
projection from $C(K)$ onto a subspace isometric to
$C(S(\wp(\N))$ where $S(\wp(\N))$ is the Stone space of
$\wp(\N)$,i.e., the Stone-$\check{\text{C}}$ech
compactification $\beta \N$ of $\N$. Hence, $C(S(\wp(\N))$ is
isometric to $\ell_\infty$. To finish the proof we observe that
$C(K_i)$ does not contain any copy of $\ell_\infty$. Indeed, a
$C(K)$ space containing a (necessarily complemented) copy of
$\ell_\infty$ obviously has hyperplanes isomorphic to the
entire space which is not the case for $C(K_i)$ by
\cite[Theorem~2.4]{Koszmider}. Therefore, we can use
Theorem~\ref{thm-l-infinity} and
Proposition~\ref{prop-suf-condition} to obtain that $C(K)$ is
extremely non-complex.
\end{proof}

\begin{remark} By Theorem~\ref{Schlackow-theorem} and the
comments below it the above space has
many operators which are not weak multipliers.
\end{remark}

\emph{Acknowledgement.} The first author expresses his gratitude to
the other authors for their support and hospitality during his visit
to the University of Granada.

\end{document}